\documentclass{amsart}

\usepackage{amsmath}
\usepackage{amssymb}
\usepackage{amsthm}
\usepackage{esint}
\usepackage{bm}
\usepackage[inline]{enumitem}
\usepackage{url}
\usepackage{color}

\theoremstyle{plain}
\newtheorem{theorem}{Theorem}[section]

\newtheorem{lemma}[theorem]{Lemma}

\theoremstyle{definition}
\newtheorem{definition}[theorem]{Definition}
\newtheorem{remark}[theorem]{Remark}

\theoremstyle{remark}
\newtheorem{example}[theorem]{Example}

\numberwithin{theorem}{section}
\numberwithin{equation}{section}

\newcommand{\R}{\mathbb{R}}

\newcommand{\diam}{\mathrm{diam}}

\newcommand{\cl}{\overline}

\newcommand{\loc}{\mathrm{loc}}

\DeclareMathOperator{\divergence}{div}
\newcommand{\laplacian}{\Delta}

\DeclareMathOperator*{\spt}{supp}
\newcommand{\capacity}{\mathrm{cap}}

\newcommand{\wkto}{\rightharpoonup}

\newcommand{\trinorm}[1]
{{
    \left\vert\kern-0.20ex\left\vert\kern-0.20ex\left\vert
    #1 
    \right\vert\kern-0.20ex\right\vert\kern-0.20ex\right\vert
}}

\newcommand{\per}{\mathrm{per}}

\begin{document}


\title[Uniformly Elliptic Equations on Domains with Capacity Density Conditions]{Uniformly Elliptic Equations on Domains with Capacity Density Conditions: 
	Existence of H\"{o}lder Continuous Solutions and Homogenization Results}
\author{Takanobu Hara}
\email{takanobu.hara.math@gmail.com}
\address{
Mathematical Institute, Tohoku University,
6-3, Aramaki Aza-Aoba, Aoba-ku, Sendai, Miyagi 980-8578, Japan }
\date{\today}


\begin{abstract}
This is a progress report on study of uniformly elliptic Poisson-type equations on domains with capacity density conditions (CDC domains).
We give a brief summary of known facts of CDC domains, including Hardy's inequality,
and review a previous work of existence of globally H\"{o}lder continuous solutions.
Additionally, we apply the result to homogenization problems of $\epsilon$-periodic coefficients and present a convergence rate estimate of $L^{\infty}$ norms.
\end{abstract}


\maketitle


\section{Introduction}\label{sec:introduction}

This paper is a progress report on study of partial differential equations of the form
\begin{align}
- \divergence(A(x) \nabla u) = \mu & \quad \text{in} \ \Omega, \label{eqn:DE}
\\
u = 0 & \quad \text{on} \ \partial \Omega. \label{eqn:BC}
\end{align}
Here, $\Omega$ is an open set in $\R^{n}$ ($n \ge 2$),
$A$ is a uniformly elliptic matrix-valued function on $\Omega$, and
$\mu$ is a linear continuous functional on $C_{c}(\Omega)$.
Below, we denote by $\mathcal{M}(\Omega)$ 
the set of all linear continuous functionals on $C_{c}(\Omega)$
and identify it as the space of the difference of two nonnegative Radon measures. 
The precise assumption on $A$ is as follows:
there are constants $0 < \lambda \le L < \infty$ such that
\begin{equation}\label{eqn:uniformly_elliptic}
\begin{aligned}
	|A(x) z| \le L |z|, \\
	A(x) z \cdot z \ge  \lambda |z|^{2}
\end{aligned}
\quad
\begin{aligned}
\forall z \in \R^{n}, \
\text{a.e.} \ x \in \Omega.
\end{aligned}
\end{equation}
We discuss globally H\"{o}lder continuous solutions to \eqref{eqn:DE}-\eqref{eqn:BC}.
Specific assumptions on $\Omega$ and $\mu$ for this purpose will be provided later.

The content of this paper is divided into two main parts.
The first part presents an existence result for globally H\"{o}lder continuous solutions to \eqref{eqn:DE}-\eqref{eqn:BC},
building upon the author's previous work \cite{Hara_2024, hara2023global}.
The second part focuses on its application to homogenization problems.
The new results introduced here will appear in a more generalized form in a forthcoming paper, and detailed proofs are therefore omitted.

Elliptic partial differential equations of divergence form are a classical but still important topic.
Even considering only the classical Laplace equation, it has a wide range of applications, including in physics and engineering.
The importance of general elliptic equations was recognized through variational problems. 
The problem \eqref{eqn:DE}-\eqref{eqn:BC} serves as a typical model of them.

Solving \eqref{eqn:DE}-\eqref{eqn:BC} in the sense of globally continuous solutions
requires certain assumptions on $\Omega$ and $\mu$.
Even if $\mu$ vanishes near the boundary, the boundary regularity of $u$ requires an exterior condition on $\Omega$.
Then, the Green function of $\Omega$ also vanishes on $\partial \Omega \times \partial \Omega$.
Consequently, a required condition on $\mu$ should differ from those used for interior regularity estimates.
However, deriving specific conditions on $\mu$ remains challenging due to the lack of quantitative estimates.

On the other hand, assuming a modulus of continuity of solutions changes the situation significantly.
Due to an argument in \cite{MR998128},
if there is an $\alpha$-H\"{o}lder continuous superharmonic function $u$ in $\Omega$, 
then $\mu := - \laplacian u$ satisfies the Morrey-type condition
\begin{equation}\label{eqn:def_norm}
\trinorm{\mu}_{\alpha, \Omega}
:=
\sup_{ \substack{x \in \Omega \\ 0 < r < \delta(x) / 2} } r^{2 - n - \alpha} |\mu|(B(x, r))
\le
C [ u ]_{\alpha, \Omega},
\end{equation}\label{eqn:hoelder_seminorm}
where $\delta$ is the Euclidean distance from $\partial \Omega$, $C$ is a constant depending on $n$ and
\begin{equation}
[u]_{\alpha, \Omega} := \sup_{ \substack{x, y \in \Omega \\ x \neq y} } \frac{|u(x) - u(y)|}{|x - y|^{\alpha}}.
\end{equation}
Thus, we define a subspace $\mathsf{M}^{\alpha}(\Omega)$ of $\mathcal{M}(\Omega)$ by
\[
\mathsf{M}^{\alpha}(\Omega) := \left\{ \mu \in \mathcal{M}(\Omega) \colon \trinorm{\mu}_{\alpha, \Omega} < \infty \right\}.
\]
Our main theorem (Theorem \ref{thm:poisson}) shows that if $\Omega$ satisfies a capacity density condition (CDC), 
then for any $\mu \in \mathsf{M}^{\alpha}(\Omega)$, there exists a weak solution $u$ to \eqref{eqn:DE}-\eqref{eqn:BC} satisfying
\begin{equation}\label{eqn:hoelder_esti}
[ u ]_{\alpha_{0}, \Omega}
\le
\frac{C}{\lambda} \diam(\Omega)^{\alpha - \alpha_{0}} \trinorm{ \mu }_{\alpha, \Omega},
\end{equation}
where $C$ and $0 < \alpha_{0} \le \alpha$ are positive constants.
This condition on $\Omega$ is optimal a certain sense (Remark \ref{rem:necessity}).

The H\"{o}lder estimate \eqref{eqn:hoelder_esti} is applicable to homogenization problems.
Homogenization theory, particularly periodic homogenization,
provides an established mathematical framework that deduces simplified macroscopic structures from systems with microscopic structures
(e.g., \cite{MR503330, MR1329546, MR3838419, MR3839345}).
Specifically, we consider $\epsilon$-parametrized problems of the form
\begin{align}
- \divergence(A_{\epsilon}(x) \nabla u_{\epsilon}) & = \mu \quad \text{in} \ \Omega, \label{eqn:DEE}
\\
u_{\epsilon} & = 0 \quad \text{on} \ \partial \Omega. \label{eqn:DEEBC}
\end{align}
If $\{ A_{\epsilon} \}$  H-converges to $A_{0}$ (cf. Definition \ref{def:h-conv}), then $u_{\epsilon} \to u_{0}$ uniformly (Theorem \ref{thm:h_conv}).
Furthermore, for $\epsilon$-periodic $A_{\epsilon}$ and smooth $\mu$,
a quantitative convergence rate estimate holds (Theorem \ref{thm:convergence_rate}).
These results show that the approximation of composite materials using the homogenization method is justified in the sense of uniform convergence of solutions, and it remains fairly robust with respect to conditions on both $\Omega$ and $\mu$.

The structure of the paper is as follows.
In Section 2, we introduce a capacity density condition and provide a brief summary of domains satisfying it.
In Section 3, we present an existence theorem of solutions to \eqref{eqn:DE}-\eqref{eqn:BC} for $\mu \in \mathsf{M}^{\alpha}(\Omega)$ under the capacity density condition.
In Section 4, we apply the results of Section 3 to abstract homogenization problems.
In Section 5, we discuss periodic homogenization and convergence rate of solutions.

We use the following notation. 
Throughout below, $\Omega$ is an open set in $\R^{n}$ ($n \geq 2$) with a nonempty boundary $\partial \Omega$.
The Euclidean distance from $\partial \Omega$ is denoted by $\delta$. 
A ball centered at $x$ with radius $r > 0$ is written as $B(x, r)$.
\begin{itemize}
\item
$\nabla u :=$ the gradient of $u$.
$\nabla^{2} u :=$ the Hessian matrix of $u$.
\item
$\divergence \bm{u} :=$ the divergence of $\bm{u}$. $\Delta u := \divergence (\nabla u) =$ the Laplacian of $u$.
\item
$C_{c}(\Omega) :=$
the set of all continuous functions with compact support in $\Omega$.
\item
$C_{c}^{\infty}(\Omega) := C_{c}(\Omega) \cap C^{\infty}(\Omega)$.
\item
$|E|$ := the Lebesgue measure of $E \subset \R^{n}$.
\item
$\int_{E} u \, dx :=$ the integral of $u$ on $E$ with respect to the Lebesgue measure.
\item
$u_{+} := \max \{ u, 0 \}$ and $u_{-} := \max\{ -u, 0 \}$.
\end{itemize}
When the Lebesgue measure must be indicate clearly, we use the letter $m$.
The set of all continuous linear functionals on $C_{c}(\Omega)$ is denoted by $\mathcal{M}(\Omega)$.
As in \cite[Chapter 3]{MR2018901}, if $\mu \in \mathcal{M}(\Omega)$, then there are nonnegative Radon measures $\mu_{+}$ and $\mu_{-}$ on $\Omega$
such that $\mu = \mu_{+} - \mu_{-}$.
We denote by $|\mu|$ the total variation of $\mu$.
The Sobolev space $H^{1}(\Omega)$ is the set of all weakly differentiable functions
$u$ on $\Omega$ such that $\| u \|_{H^{1}(\Omega)}^{2} = \int_{\Omega} |u|^{2} + |\nabla u|^{2} \, dx$ is finite.
We denote by $H_{0}^{1}(\Omega)$ the closure of $C_{c}^{\infty}$ in $H^{1}(\Omega)$.
The space $H^{-1}(\Omega)$ is the dual of $H_{0}^{1}(\Omega)$.
For a bounded set $E \subset \R^{n}$ and $0 < \alpha \le 1$,
we define the seminorm $[u]_{\alpha, E}$ of a function $u$ on $E$ by \eqref{eqn:hoelder_seminorm}.
We define the H\"{o}lder space $C^{\alpha}(E)$ and its norm by $\| u \|_{C^{\alpha}(E)} := \| u \|_{L^{\infty}(E)} + \diam(E)^{\alpha} [u]_{\alpha, E}$.
The letters $C$ and $c$ denotes various constants with and without indices.

\section{Capacity density condition}\label{sec:CDC}

For an open set $U \subset \R^{n}$ and a compact set $K \subset U$, 
we define the \textit{variational capacity} $\capacity(K, U)$ of the \textit{condenser} $(K, U)$ by
\begin{equation*}\label{eqn:variational_capacity}
\capacity(K, U)
:=
\inf
\left\{
\int_{\R^{n}} |\nabla u|^{2} \, dx \colon u \in C_{c}^{\infty}(U), \ u \ge 1 \, \text{on} \, K
\right\}.
\end{equation*}
We refer the reader to \cite{MR1439503, MR1801253, MR2305115} for basic facts on capacities.

\begin{definition}
We say that an open set $\Omega$ is a \textit{CDC domain} if the following capacity density condition holds:
\begin{equation}\label{eqn:CDC}
\exists \gamma > 0 \quad \text{s.t.} \quad
\frac{ \capacity( \cl{B(\xi, R)} \setminus \Omega, B(\xi, 2R)) }{ \capacity( \cl{B(\xi, R)}, B(\xi, 2R)) } \ge \gamma \quad \forall R > 0, \ \forall \xi \in \partial \Omega.
\end{equation}
\end{definition}

It is difficult to specifically calculate the variational capacity of a given condenser.
However, we can check \eqref{eqn:CDC} by the following geometric arguments.
(i)
We say that an open set $\Omega$ satisfies an \textit{exterior corkscrew condition} 
if there exists $0 < c < 1$, such that, for every $\xi \in \Omega$ and every $R > 0$, there exists a ball $B(x, cR) \subset B(\xi, R)$
such that $B(x, cR) \subset \R^{n} \setminus \Omega$.
Clearly, an exterior corkscrew condition is sufficient for the \textit{volume density condition} 
\begin{equation}\label{eqn:VDC}
| \cl{B(\xi, R)} \setminus \Omega | \ge \frac{1}{C}  R^{n} \quad \forall R > 0, \ \forall \xi \in \partial \Omega.
\end{equation}
Using a bump function, we can check that \eqref{eqn:VDC} is sufficient for \eqref{eqn:CDC}.
(ii) The following sufficient condition is giving a refinement of \eqref{eqn:VDC}.
For $E \subset \R^{n}$ and $0 \le s \le n$, the \textit{$s$-dimensional Hausdorff contents} $\mathcal{H}^{s}_{\infty}(E)$ of $E$ is defined by
\[
\mathcal{H}^{s}_{\infty}(E)
:=
\inf \left\{
\sum_{i = 1}^{\infty} r_{i}^{s} \colon E \subset \bigcup_{i = 1}^{\infty} B(x_{i}, r_{i}), r_{i} > 0
\right\}.
\]
As in \cite[Chapter 6]{MR4306765}, if there exist $n - 2 < s \le n$ and $C > 0$ such that
\begin{equation}\label{eqn:HCDC}
\mathcal{H}^{s}_{\infty}( \cl{B(\xi, R)} \setminus \Omega ) \ge \frac{1}{C}  R^{s} \quad \forall R > 0, \ \forall \xi \in \partial \Omega,
\end{equation}
then \eqref{eqn:CDC} holds.
(iii) In the case $n = 2$, \eqref{eqn:CDC} can be verified by using a more geometric approach.
A nonempty set $E$ is said to be \textit{uniformly perfect} if there exists $0 < c < 1$ such that 
\[
E \cap (B(x, R) \setminus B(x, cR)) \neq \emptyset \quad 0 < \forall R < \diam(E), \ \forall x \in E.
\]
For $n = 2$, \eqref{eqn:CDC} holds if and only if $E = \R^{2} \setminus \Omega$ is uniformly perfect and unbounded.
The survey \cite{MR2019172} contains many examples of uniformly perfect sets and various related results. 

\begin{example}\label{ex:lip}
Any bounded Lipschitz domain satisfies an exterior corkscrew condition.
Therefore, \eqref{eqn:VDC} and \eqref{eqn:CDC} are satisfied, and it is a CDC domain.
\end{example}

\begin{example}\label{ex:koch}
Even for fractals, there are cases where \eqref{eqn:CDC} can be easily checked by self-similarity.
For example, the inside of the Koch snowflake is a CDC domain.
This domain satisfies an exterior corkscrew condition.
\end{example}

\begin{example}\label{ex:slit}
Assume that $D \subset \R^{n - 1}$ satisfies \eqref{eqn:VDC}.
Then, $\Omega = B(0, 1) \setminus \{ (x', 0) \colon x' \in \R^{n - 1} \setminus D \} \subset \R^{n}$ satisfies \eqref{eqn:CDC},
because \eqref{eqn:HCDC} is satisfied with $s = n - 1 > n - 2$.
\end{example}

\begin{example}
If $\Omega$ is a bounded simply connected open set in $\R^{2}$, then $\Omega$ is a CDC domain.
In this case, $E = \R^{2} \setminus \Omega$ is unbounded and uniformly perfect.
\end{example}

Let us briefly discuss domains that are not CDC domains.
A domain with removable singularities is a typical example of a domain for which \eqref{eqn:CDC} does not hold.
A concrete instance is Zaremba's punctured ball $\{ x \in \R^{2} \colon 0 < |x| < 1 \}$.
In this simple case, $\capacity(\{ 0 \}, B(0, r)) = 0$ for all $r > 0$.
Since various results are derived from \eqref{eqn:CDC}, it follows that any domain that does not satisfy one of these results is not a CDC domain.
For example, a domain with a Dirichlet irregular point is not a CDC domain.
See, e.g., \cite[Chapter 6]{MR1801253} for specific examples.

The scale-invariant density condition \eqref{eqn:CDC} in terms of capacity is important in study of the Dirichlet problem
\begin{align*}
- \laplacian u_{g} = 0 & \quad \text{in} \ \Omega, 
\\
u_{g} = g & \quad \text{on} \ \partial \Omega. 
\end{align*}
One of typical application is boundary H\"{o}lder estimate of solutions.
If $\Omega$ is a bounded CDC domain, we can prove that the operator
\begin{equation}\label{eqn:harmonic_extention}
C^{\alpha}(\partial \Omega) \ni g \mapsto u_{g} \in C^{\alpha}(\cl{\Omega})
\end{equation}
is bounded, where $0 < \alpha < 1$ is a positive constant depending on $n$ and $\gamma$. 
The detailed proof can be found in \cite{MR163053} and \cite[Chapter 6]{MR2305115}.
Aikawa \cite[Theorem 3]{MR1924196} showed a converse statement under a qualitative assumption;
if $\Omega$ is a Dirichlet regular domain, then the boundedness of \eqref{eqn:harmonic_extention} implies \eqref{eqn:CDC}.

A significant property of CDC domains is Hardy's inequality;
if \eqref{eqn:CDC} holds, then there exists a positive constant $c_{H}$ depending only on $n$ and $\gamma$ such that
\begin{equation}\label{eqn:hardy}
c_{H} \int_{\Omega} \left( \frac{\varphi}{\delta} \right)^{2} \, dx \le \int_{\Omega} |\nabla \varphi|^{2} \, dx \quad \forall \varphi \in C_{c}^{\infty}(\Omega).
\end{equation}
If $\Omega$ is a bounded Lipschitz domain, this inequality is obtained as
a patchwork of one-dimensional Hardy inequalities (e.g., \cite{MR163054}).
However, such a strategy does not work for a CDC domain because it cannot be expressed as the union of graphs (see, Examples \ref{ex:koch} and \ref{ex:slit}).
Three different proofs were proposed by Ancona \cite{MR856511}, Lewis \cite{MR946438} and Wannebo \cite{MR1010807}.
Lewis established $L^{p}$-Hardy inequalities for $1 \le p < \infty$, and  Wannebo proved weighted $L^{p}$-Hardy inequalities.
Later, Mikkonen \cite{MR1386213} proved more general weighted $L^{p}$-Hardy inequalities using Lewis's strategy and a new technique.
Research in this area has since grown significantly, and the monograph \cite{MR4306765} is a good reference for the current state of the field.
If $n = 2$, \eqref{eqn:hardy} is also equivalent to \eqref{eqn:CDC} (see, \cite[Theorem 2]{MR856511} and \cite[Theorem 7.20]{MR4306765}).
This equivalence is not true for $n \ge 3$.

Ancona's proof of \eqref{eqn:hardy} differs from the others.
Through a variational argument, the validity of \eqref{eqn:hardy} is equivalent to the existence of a positive function $U$ satisfying
\begin{equation}\label{eqn:sb}
\laplacian U + c_{H} \frac{U}{\delta^{2}} \le 0 \quad \text{in} \ \Omega.
\end{equation}
In \cite[Theorem 1]{MR856511}, it was proved that if \eqref{eqn:CDC} holds, then there exists $U$ on $\Omega$ satisfying \eqref{eqn:sb} and
\begin{equation}\label{eqn:sb2}
\frac{1}{C} \delta(x)^{\alpha} \le U(x) \le C \delta(x)^{\alpha} \quad \forall x \in \Omega,
\end{equation}
where $C > 0$ and $0 < \alpha \le 1$ are constants.
A function $U$ satisfying \eqref{eqn:sb}-\eqref{eqn:sb2} is called a \textit{strong barrier}.
In any dimension, the existence of a strong barrier $U$ on $\Omega$ implies $\eqref{eqn:CDC}$ (\cite[Theorem 2']{MR856511}).
A strong barrier has applications to the existence problem of solutions to Poisson's equation, which will be discussed in the next section, and is interesting in itself.
For the proof of weighted $L^{p}$-Hardy inequality based on this strategy, see \cite{Hara_2024}.

\section{Uniformly elliptic equations on CDC domains}

Let $A(x) \in L^{\infty}(\Omega)^{n \times n}$, and let $\mu \in \mathcal{M}(\Omega)$.
We understand \eqref{eqn:DE} in the sense of distributions;
a function $u \in H^{1}_{\loc}(\Omega)$ is a weak solution to \eqref{eqn:DE} if
\[
\int_{\Omega} A(x) \nabla u \cdot \nabla \varphi \, dx = \int_{\Omega} \varphi \, d \mu
\]
for all $\varphi \in C_{c}^{\infty}(\Omega)$.

Considering \eqref{eqn:DE}-\eqref{eqn:BC} for $\mu \in \mathsf{M}^{\alpha}(\Omega)$
raises problems of both existence and regularity of solutions.
If $\mu \in \mathsf{M}^{\alpha}(\Omega)$ and $\mu_{\pm}(\Omega)$ are finite, then $\mu \in H^{-1}(\Omega)$.
Unfortunately, there is no inclusion between $\mathsf{M}^{\alpha}(\Omega)$ and $H^{-1}(\Omega)$.
It is not obvious whether there is a solution of \eqref{eqn:DE}-\eqref{eqn:BC} for $\mu \in \mathsf{M}^{\alpha}(\Omega)$,
and even if there is, it may not belong to $H_{0}^{1}(\Omega)$.
If $u \in H^{1}_{\loc}(\Omega) $ satisfies \eqref{eqn:DE},
then $u \in C^{\alpha_{0}}_{\loc}(\Omega)$ for some $\alpha_{0} > 0$ (e.g., \cite{MR0271383, MR0964029, MR998128, MR1354887, MR1461542}).
However, boundary regularity of it is not clear
because known regularity estimates assume that $u \in H^{1}(\Omega)$ and $\mu_{\pm}(\Omega)$ are finite.

These two problems can be solved if there exists a supersolution controlled by a known function. 
This is a standard way to solve boundary value problems, and it can be constructed in our setting
by combining a method in \cite{MR856511} and a boundary H\"{o}lder estimate in \cite{MR1461542}.
More precisely, we patchwork infinitely many auxiliary functions controlled by the H\"{o}lder estimate.
The result is as follows.

\begin{theorem}[{\cite[Theorem 1.2]{hara2023global}}]
Let $\Omega$ be a bounded CDC domain.
Assume that $A$ satisfis \eqref{eqn:uniformly_elliptic}.
Then, for any $0 \le \mu \in \mathsf{M}^{\alpha}(\Omega)$, there exists $U \in H^{1}_{\loc}(\Omega) \cap C(\Omega)$ satisfying
\begin{equation*}\label{eqn:barrier}
- \divergence \left( A(x) \nabla U \right) \ge \mu \quad \text{in} \ \Omega,
\end{equation*}
\begin{equation*}\label{eqn:bound_of_barrier}
\begin{split}
& 
\frac{1}{\lambda C} \diam(\Omega)^{\alpha - \alpha_{0}} \trinorm{\mu}_{\alpha, \Omega} \, \delta(x)^{ \alpha_{0} } \\
& \quad
\le
U(x)
\le
\frac{C}{\lambda} \diam(\Omega)^{\alpha - \alpha_{0}} \trinorm{\mu}_{\alpha, \Omega} \, \delta(x)^{\alpha_{0}}
\quad
\forall x \in \Omega,
\end{split}
\end{equation*}
where $C$ and $\alpha_{0}$ are positive constants depending only on $n$, $L / \lambda$, $\alpha$ and $\gamma$. 
\end{theorem}

From the comparison principle, the following existence result follows.

\begin{theorem}[{\cite[Theorem 1.3]{hara2023global}}]\label{thm:poisson}
Let $\Omega$ be a bounded CDC domain.
Assume that $A$ satisfies \eqref{eqn:uniformly_elliptic}.
Then, for any $\mu \in \mathsf{M}^{\alpha}(\Omega)$,
there exists a unique weak solution
$u \in H^{1}_{\loc}(\Omega) \cap C(\cl{\Omega})$ to \eqref{eqn:DE}-\eqref{eqn:BC}.
Moreover, \eqref{eqn:hoelder_esti} holds with some positive constants $C$ and $\alpha_{0}$, 
depending only on $n$, $L / \lambda$, $\alpha$ and $\gamma$.
\end{theorem}

\begin{remark}\label{rem:necessity}
If there is a pair of positive constants $C$ and $\alpha_{0}$ satisfying \eqref{eqn:hoelder_esti},
then $\gamma$ in \eqref{eqn:CDC} is estimated from below.
The precise statement is as follows.
Assume that there are constants $C > 0$ and $0 < \alpha_{0} \le 1$, such that,
for any $\xi \in \partial \Omega$, $0 < R < \diam(\Omega)$ and $0 \le \mu \in \mathsf{M}^{1}(\Omega \cap B(\xi, R))$,
there exists a nonnegative weak supersolution $u \in H^{1}_{\loc}(\Omega \cap B(\xi, R))$ to
$- \laplacian u = \mu$ in $\Omega \cap B(\xi, R)$
satisfying
\[
\left\| \frac{u}{\delta^{\alpha_{0}}} \right\|_{L^{\infty}(\Omega \cap B(\xi, R))} \le C R^{1 - \alpha_{0}} \trinorm{ \mu }_{1, \Omega \cap B(\xi, R)}.
\]
Then, there exists a positive constant $\gamma$ depending only on $n$, $C$ and $\alpha_{0}$
such that \eqref{eqn:CDC} holds.
This follows from an argument similar to \cite[Theorem 2']{MR856511}.
\end{remark}

\begin{remark}
If $\mu$ vanishes near $\partial \Omega$, then we can take $\alpha_{0} = \alpha$ under additional assumptions on $A$ and $\Omega$.
For example, assume that $A \in C^{0, \beta}(\cl{\Omega})$ and $\Omega$ is a $C^{1, \beta}$-domain ($0 < \beta < 1$).
Then, we can prove that $u \in C^{\alpha}_{\loc}(\Omega) \cap C^{1, \beta}(\cl{\Omega} \setminus \spt \mu)$
by combining \cite[Theorem 5.12]{MR2286361} and \cite[Theorem 1.6]{MR1290667}.
Therefore, $u \in C^{\alpha}(\cl{\Omega})$.

Note that boundary behavior of harmonic functions on a polygon is depends on the angles
and that the structure condition \eqref{eqn:uniformly_elliptic} is invariant, up to a difference in constants, under bi-Lipschitz transformations.
Hence, additional assumptions must be made for both $A$ and $\Omega$.
For more specific details, see, \cite[Chapter 6]{MR1349825} for instance.
\end{remark}

Finally, we consider examples of $\mu \in \mathsf{M}^{\alpha}(\Omega)$.

\begin{example}
Assume that $\mu$ is absolutely continuous with respect to the Lebesgue measure $m$.
By the Radon-Nikod\'{y}m theorem, there exists $f \in L^{1}_{\loc}(\Omega)$ such that
\begin{equation}\label{eqn:density_function} 
\mu = f m.
\end{equation}
Assume also that there are $1 \le q \le \infty$ and $0 < \alpha \le 1$ such that
\begin{equation}\label{eqn:norm_of_f}
M := \sup_{ \substack{x \in \Omega \\ 0 < r < \delta(x) / 2} } r^{2 - \alpha - n / q} \| f \|_{L^{q}(B(x, r))} < \infty.
\end{equation}
Then, by H\"{o}lder's inequality, $\trinorm{\mu}_{\alpha, \Omega} \le M$.

Following \cite{MR856511}, let us assume that $f \delta^{t} \in L^{q}(\Omega)$ for some $n / 2 < q \le \infty$ and $0 \le t < 2 - n / q$.
Then $M$ in \eqref{eqn:norm_of_f} is finite for $\alpha = \min\{ 2 - n / q - t, 1\}$.
Indeed, for any $x \in \Omega$ and any $0 < r \le \delta(x) / 2$, we have
\[
\begin{split}
\int_{B(x, r)} |f|^{q} \, dx
& =
\int_{B(x, r)} |f|^{q} \delta^{q (2 - \alpha - n / q)} \delta^{-q (2 - \alpha - n / q)} \, dx
\\
& \le
\| f \delta^{t} \|_{L^{q}(\Omega)}^{q} \, r^{-q (2 - \alpha - n / q)}.
\end{split}
\]
\end{example}

\begin{example}
Let $\sigma$ be a Radon measure on $\Omega$ satisfying
\[
M := \sup_{ \substack{x \in \Omega \\ 0 < r < \infty} } r^{1 - n} \sigma(\Omega \cap B(x, r)) < \infty.
\]
Take a function $f$ on $\Omega$ satisfying $|f(x)| \le C \delta(x)^{\alpha - 1}$ for all $x \in \Omega$, where $C > 0$ and $0 < \alpha \le 1$.
Then, $f \sigma \in \mathsf{M}^{\alpha}(\Omega)$ and $\trinorm{ f \sigma }_{\alpha, \Omega} \le C M$.
Note $\sigma \in \mathsf{M}^{1}(\Omega)$ by definition.
This example is useful when dealing with Neumann boundary data.

A more concrete example is as follows.
Let $n \ge 3$ and $0 < \alpha < 1$.
For $0 < R < 1$, define a spherically symmetric function $u_{R}(|x|) = u_{R}(r)$ on $B(0, 1)$ by
\[
u_{R}(r)
:=
\begin{cases}
\displaystyle
\frac{(1 - R)^{\alpha}}{ R^{2 - n} - 1 } \left( r^{2 - n} - 1\right) & \text{if} \  R < r < 1,
\\
(1 - R)^{\alpha} & \text{otherwise.}
\end{cases}
\]
Then, $- \laplacian u_{R} = f(R) \mathcal{H}^{n - 1} \lfloor_{ \{ |x| = R \} }$ and $f(R) / (1 - R)^{\alpha - 1}$ is bounded,
where $\mathcal{H}^{n - 1}$ is the $(n - 1)$-dimensional Hausdorff measure.
This $f$ and $\sigma = \mathcal{H}^{n - 1} \lfloor_{ \{ |x| = R \} }$ satisfy the above conditions.
Through simple calculations, we can check that
\[
\| u_{R} \|_{C^{\alpha}(\cl{ B(0, 1) })} = 1
\]
and
\[
\| \nabla u_{R} \|_{L^{2}(B(0, 1))}^{2}
=
(n - 2) \frac{ (1 - R)^{2 \alpha} }{ R^{2 - n} - 1 }.
\]
When $0 < \alpha < 1 / 2$, the latter tends to $+ \infty$ as $R \to 1$. 
This example shows that the use of $H_{0}^{1}(\Omega)$ is not appropriate to consider H\"{o}lder continuous solutions.
\end{example}

\begin{example}
Let $u$ be a superharmonic function in $C^{\alpha}(\cl{\Omega})$,
and let $\phi \in C^{\beta}(\R)$ be concave on the range of $u$.
Then, $\mu := - \laplacian \phi(u)$ belongs to $\mathsf{M}^{\alpha \beta}(\Omega)$.
This argument is effective when addressing semilinear problems.
\end{example}


\section{Abstract homogenization}\label{sec:h-conv}

Let us recall the definition of H-convergence of matrix-valued functions.
Define a class of matrix-valued functions $M(\lambda, L, \Omega)$ by
\[
M(\lambda, L, \Omega)
:=
\left\{
A \in L^{\infty}(\Omega)^{n \times n} \colon \text{$A$ satisfies \eqref{eqn:uniformly_elliptic_h}}
\right\},
\]
where
\begin{equation}\label{eqn:uniformly_elliptic_h}
\begin{aligned}
	|A(x) z |^{2} \le L A(x) z \cdot z, \\
	A(x) z \cdot z \ge  \lambda |z|^{2}
\end{aligned}
\quad
\begin{aligned}
\forall z \in \R^{n}, \
\text{a.e.} \ x \in \Omega.
\end{aligned}
\end{equation}
If a symmetric matrix-valued function $A(x)$ satisfies \eqref{eqn:uniformly_elliptic}, then $A \in M(\lambda, L, \Omega)$.
Furthermore, \eqref{eqn:uniformly_elliptic} implies $A \in M(\lambda, L^{2} / \lambda, \Omega)$ even if it is not symmetric.

\begin{definition}\label{def:h-conv}
Let $\{ A_{\epsilon} \}_{\epsilon > 0} \subset M(\lambda, L, \Omega)$,
and let $A_{0} \in L^{\infty}(\Omega)^{n \times n}$.
We say that $\{ A_{\epsilon} \}$ \textit{H-converges} to $A_{0}$ ($A_{\epsilon} \xrightarrow{H} A_{0}$)
if for every $\mu \in H^{-1}(\Omega)$,
the sequence $\{ u_{\epsilon} \} \subset H_{0}^{1}(\Omega)$ of weak solutions to \eqref{eqn:DEE}-\eqref{eqn:DEEBC} satisfies
\begin{align*}
u_{\epsilon} & \wkto u_{0} \ \text{weakly in} \ H_{0}^{1}(\Omega),
\\
A_{\epsilon} \nabla u_{\epsilon} & \wkto A_{0} \nabla u_{0} \  \text{weakly in} \ L^{2}(\Omega)^{n},
\end{align*}
where $u_{0} \in H_{0}^{1}(\Omega)$ is the weak solution to \eqref{eqn:DEE}-\eqref{eqn:DEEBC} with respect to $\epsilon = 0$.
\end{definition}

For basics of H-convergence, we refer the reader to \cite{MR1493039, MR1329546, MR1859696}.
If $\{ A_{\epsilon} \}_{\epsilon > 0} \subset M(\lambda, L, \Omega)$ and $A_{\epsilon} \xrightarrow{H} A_{0}$,
then $A_{0} \in M(\lambda, L, \Omega)$.
If $A_{\epsilon}(x) \to A_{0}(x)$ a.e. $x \in \Omega$,
then, $A_{\epsilon} \xrightarrow{H} A_{0}$.
One nontrivial example is $\epsilon$-periodic coefficients, which will be discussed in the next section.

The following abstract result is a direct consequence of Theorem \ref{thm:poisson}.

\begin{theorem}\label{thm:h_conv}
Let $\Omega$ be a bounded CDC domain.
Assume that $A_{\epsilon} \xrightarrow{H} A_{0}$.
For $\mu \in \mathsf{M}^{\alpha}(\Omega)$,
let $u_{\epsilon} \in H_{\loc}^{1}(\Omega) \cap C(\cl{\Omega})$ be the weak solution to \eqref{eqn:DEE}-\eqref{eqn:DEEBC}.
Then, $u_{\epsilon} \to u_{0}$ uniformly in $\Omega$,
where $u_{0} \in H^{1}_{\loc}(\Omega) \cap C(\cl{\Omega})$ is the weak solution to \eqref{eqn:DEE}-\eqref{eqn:DEEBC} with respect to $\epsilon = 0$.
\end{theorem}

\begin{proof}
For $k \ge 1$, we define $\Omega_{k} = \{ x \in \Omega \colon \delta(x) > 1 / k \}$. Let $\mu_{k} := \mathbf{1}_{\Omega_{k}} \mu$.
Since $(\mu_{k})_{\pm}(\Omega)$ are finite, $\mu_{k} \in H^{-1}(\Omega)$.
Take any $x \in \Omega$ and any $0 < r \le \delta(x) / 2$.
If $\delta(x) \le 2 / k$, then we have
\[
\begin{split}
|\mu - \mu_{k}|(B(x, r))
& \le
\trinorm{\mu}_{\alpha, \Omega} r^{n - 2 + \alpha}
\\
& \le
\trinorm{\mu}_{\alpha, \Omega} r^{n - 2 + \alpha / 2} k^{- \alpha / 2}.
\end{split}
\]
Since $|\mu - \mu_{k}|(B(x, r)) = 0$ for $\delta(x) > 2 / k$, we also get
\[
\trinorm{ \mu - \mu_{k} }_{\alpha / 2, \Omega}
\le
\trinorm{\mu}_{\alpha, \Omega} k^{- \alpha / 2}.
\]
For $\epsilon \ge 0$, let $u_{\epsilon}^{k} \in H_{0}^{1}(\Omega)$ be the weak solution to 
\begin{align*}
- \divergence(A_{\epsilon}(x) \nabla u_{\epsilon}^{k}) & = \mu_{k} \quad \text{in} \ \Omega,
\\
u_{\epsilon}^{k} & = 0 \quad \text{on} \ \partial \Omega.
\end{align*}
By Theorem \ref{thm:poisson}, there exists a constant $C$, independent of $\epsilon$, such that
\[
\begin{split}
\| u_{\epsilon} - u_{0} \|_{L^{\infty}(\Omega)}
& \le
\| u_{\epsilon} - u_{\epsilon}^{k} \|_{L^{\infty}(\Omega)}
+
\| u_{\epsilon}^{k} - u_{0}^{k} \|_{L^{\infty}(\Omega)}
+
\| u_{0}^{k} - u_{0} \|_{L^{\infty}(\Omega)}
\\
& \le 
\| u_{\epsilon}^{k} - u_{0}^{k} \|_{L^{\infty}(\Omega)} 
+
\frac{2C}{\lambda} \trinorm{ \mu - \mu_{k} }_{\alpha / 2, \Omega}
\\
& \le 
\| u_{\epsilon}^{k} - u_{0}^{k} \|_{L^{\infty}(\Omega)} 
+
\frac{2C}{\lambda} \trinorm{\mu}_{\alpha, \Omega} k^{-\alpha / 2}.
\end{split}
\]
By assumption, $u_{\epsilon}^{k} \wkto u_{0}^{k}$ weakly in $H_{0}^{1}(\Omega)$.
Meanwhile, Theorem \ref{thm:poisson} provides
\[
\| u_{\epsilon}^{k} \|_{C^{\alpha_{0}}(\cl{\Omega})} \le C \trinorm{\mu_{k} }_{\alpha, \Omega}
\]
for all $\epsilon > 0$.
It follows from the Ascoli-Arzel\`{a} theorem that $u_{\epsilon}^{k} \to u_{0}$ uniformly.
Consequently, we obtain
\[
\limsup_{\epsilon \to 0} \| u_{\epsilon} - u_{0} \|_{L^{\infty}(\Omega)} \le \frac{2C}{\lambda} \trinorm{\mu}_{\alpha, \Omega} k^{-\alpha / 2}.
\]
Since $k$ can be chosen arbitrarily large, the left-hand side is zero.
\end{proof}

\section{Periodic homogenization}

To address periodic homogenization, we introduce additional notation.
Let $Y = (0, 1)^{n}$, and let $\bm{e}_{i}$ be the $i$-th unit vector in $\R^{n}$.
For $0 \le \alpha \le \infty$, we denote by $C^{\alpha}_{\per}(\R^{n})$ the set of all $Y$-periodic $C^{\alpha}$-functions on $\R^{n}$.
Here, $Y$-periodic means that $f(x + \bm{e}_{i}) = f(x)$ for all $x \in \R^{n}$ and all $1 \le i \le n$.
We often identify the set of $Y$-periodic functions on $\R^{n}$ with the set of functions on $Y$.
The Sobolev space $H^{1}_{\per}(Y)$ is the closure of $C^{\infty}_{\per}(Y)$ with respect to the $H^{1}(Y)$ norm.
We also set $H^{1}_{\per, 0}(Y) := \{ f \in H^{1}_{\per}(Y) | \int_{Y} f \, dy = 0 \}$.

Let $A(y)$ be a $Y$-periodic matrix-valued function in $M(\lambda, L)(\R^{n})$, and let
\begin{equation}\label{eqn:periodic_matrix}
A_{\epsilon}(x) = A \left( \frac{x}{\epsilon} \right).
\end{equation}
In this case, $A_{\epsilon} \xrightarrow{H} A_{0}$, where $A_{0}$ is a constant matrix which is given by
\begin{equation}\label{eqn:homogenized_coef}
A_{0} \bm{e}_{i}
=
\int_{Y} A(y) (\bm{e}_{i} + \nabla_{y} \chi_{i}(y)) \, dy
\end{equation}
and
\begin{equation}\label{eqn:corrector}
\begin{cases}
\int_{Y}
A(y) \left( \bm{e}_{i} + \nabla_{y} \chi_{i} \right)
\cdot \nabla_{y} \varphi
\, dy = 0
\quad
\forall \varphi \in H^{1}_{\per}(Y),
\\
\chi_{i} \in H^{1}_{\per, 0}(Y).
\end{cases}
\end{equation}
For details, see for example, \cite{MR503330, MR1329546, MR3839345}.

By Theorem \ref{thm:h_conv}, we have the following.

\begin{theorem}\label{thm:periodic}
Let $\Omega$ be a bounded CDC domain.
Define $\{ A_{\epsilon} \}_{\epsilon > 0} \subset M(\lambda, L, \Omega)$ by \eqref{eqn:periodic_matrix}.
For $\mu \in \mathsf{M}^{\alpha}(\Omega)$,
let $u_{\epsilon} \in H_{\loc}^{1}(\Omega) \cap C(\cl{\Omega})$ be the weak solution to \eqref{eqn:DEE}-\eqref{eqn:DEEBC}.
Then, $u_{\epsilon} \to u_{0}$ uniformly in $\Omega$,
where $u_{0} \in H^{1}_{\loc}(\Omega) \cap C(\cl{\Omega})$ is the weak solution to \eqref{eqn:DEE}-\eqref{eqn:DEEBC} 
with respect to \eqref{eqn:homogenized_coef}-\eqref{eqn:corrector}.
\end{theorem}

Under a further interior regularity assumption on $\mu$, we can prove the following quantitative result.

\begin{theorem}\label{thm:convergence_rate}
Suppose the assumptions in Theorem \ref{thm:periodic} hold.
Assume further that $\mu$ is given by \eqref{eqn:density_function} and $f$ satisfies \eqref{eqn:norm_of_f} for some $n < q \le \infty$ and $0 < \alpha \le 1$.
Then, for all $0 < \epsilon \le 1$, we have
\[
\| u_{\epsilon} - u_{0} \|_{L^{\infty}(\Omega)} \le \frac{C}{\lambda} M \epsilon^{\alpha_{0} / 2},
\]
where $C$ is a constant depending only on $n$, $L / \lambda$, $\alpha$, $\gamma$, $q$ and $\diam(\Omega)$.
\end{theorem}

The proof of Theorem \ref{thm:convergence_rate} is divided into the following four steps.
\textbf{(i)} Check the regularity of $u_{0}$.
\textbf{(ii)} Subtract an auxiliary function from $u_{\epsilon} - u_{0}$ and prove an energy estimate for it.
\textbf{(iii)} Derive an interior $L^{\infty}$ estimate by using the energy estimate and an iteration technique in the De Giorgi-Nash-Moser theory.
\textbf{(iv)} Combine the interior $L^{\infty}$ estimate with Theorem \ref{thm:poisson} to obtain the desired global estimate.

Step \textbf{(ii)} above is a standard strategy.
However, for the discussion in Step \textbf{(iii)}, some adaptation of known results is necessary.
Therefore, we repeat the proof using the following two lemmas. 
Lemma \ref{lem:corrector} follows from a standard regularity result (e.g., \cite[Theorem 8.22]{MR1814364}),
and Lemma \ref{lem:vec_pot} is derived through further arguments based on Lemma \ref{lem:corrector}.

\begin{lemma}\label{lem:corrector}
Let $\chi_{i}$ be a weak solution to \eqref{eqn:corrector}.
Then, there are positive constants $C$ and $\alpha$ depending only on $n$ and $L / \lambda$
such that 
\begin{equation*}\label{eqn:regularity_of_chi}
\| \chi_{i} \|_{C^{\alpha}_{\per}(Y)} \le C.
\end{equation*}
\end{lemma}

\begin{lemma}[{\cite[Proposition 3.11]{MR3838419}}]\label{lem:vec_pot}
Define $\bm{d}_{i} = (d_{i 1}, \cdots d_{i n}) \in L^{2}(Y)^{n}$  by
\begin{equation*}\label{eqn:def_of_d}
\bm{d}_{i}(y)
:=
A(y) \left( \bm{e}_{i} + \nabla_{y} \chi_{i}(y) \right) - \bm{c},
\quad
\bm{c} := \int_{Y} A \left( \bm{e}_{i} + \nabla_{y} \chi_{i} \right) \, dy.
\end{equation*}
Then, there exists $V_{i} = (V_{ijk}) \in H^{1}_{\per}(Y)^{n \times n}$ such that 
\begin{gather*}
V_{ijk} = - V_{ikj} \quad \text{for all $1 \le j, k \le n$}, \label{eqn:vec_pot_skew}
\\
 \sum_{k = 1}^{n} \frac{\partial V_{ijk}}{\partial x_{k}} = d_{ij}, \label{eqn:vec_pot}
\\
\| V_{i} \|_{L^{\infty}(Y)} \le C L, \label{eqn:vec_pot_bound}
\end{gather*}
where $C$ is a constant depending only on $n$ and $L / \lambda$.
\end{lemma}

As a consequence of the regularity estimates,
\[
w_{\epsilon} := u_{\epsilon} - u_{0} - \epsilon \sum_{i = 1}^{n} \chi_{i}\left( \frac{\cdot}{\epsilon} \right) \frac{\partial u_{0}}{\partial x_{i}}
\]
satisfies
\begin{equation*}\label{eqn:step2}
\left\| w_{\epsilon} \right\|_{L^{\infty}(\Omega_{R})}
\le
\| w_{\epsilon} \|_{L^{\infty}(\partial \Omega_{R})}
+
C \epsilon \| \nabla^{2} u_{0} \|_{L^{q}( \Omega_{R} )},
\end{equation*}
where $\Omega_{R} = \{ x \in \Omega \colon \delta(x) > R \}$.
Meanwhile, by Theorem \ref{thm:poisson} and assumption on $f$,
\begin{equation*}\label{eqn:bound_for_d2u0}
\| \nabla^{2} u_{0} \|_{L^{q}(\Omega_{R})} \le \frac{C}{\lambda} M R^{\alpha_{0} - 2}
\end{equation*}
and
\begin{equation*}\label{eqn:bound_for_du0}
\| \nabla u_{0} \|_{L^{\infty}(\Omega_{R})} \le \frac{C}{\lambda} M R^{\alpha_{0} - 1}.
\end{equation*}
Consequently, we obtain
\[
\| u_{\epsilon} - u_{0} \|_{L^{\infty}(\Omega_{R})}
\le
\frac{C}{\lambda} M \left( R^{\alpha_{0}} + \epsilon R^{\alpha_{0} - 1} + \epsilon R^{\alpha_{0} - 2} \right).
\]
Setting $R = \sqrt{\epsilon}$ and using Theorem \ref{thm:poisson} again, we arrive at the desired estimate.

\section{Conclusion}

We considered globally H\"{o}lder continuous solutions to \eqref{eqn:DE}-\eqref{eqn:BC}.
If $\Omega$ satisfies \eqref{eqn:CDC} and $\mu$ satisfies \eqref{eqn:def_norm},
then there exists a globally H\"{o}lder continuous solution $u$ to \eqref{eqn:DE}-\eqref{eqn:BC}.
Several examples of $\Omega$ and $\mu$ are discussed.
Under these assumptions, we showed that H-convergence of $\{ A_{\epsilon} \}$ leads to uniform convergence of $\{ u_{\epsilon} \}$.
For $\epsilon$-periodic coefficients, a quantitative convergence rate estimate holds under an additional assumption on $\mu$.
In these applications, uniform global H\"{o}lder estimates for solutions played a crucial role.

\section*{acknowledgement}
This work was supported by
JSPS KAKENHI (doi:10.13039/501100001691) Grant Numbers JP17H01092 and JP23H03798 
and JST CREST (doi:10.13039/501100003382) Grant Number JPMJCR18K3.
The author appreciates the helpful comments received during presentations at Meijo University, the University of Tokyo, and Nagoya University.


\bibliographystyle{abbrv}
\bibliography{reference}


\end{document}